\documentclass[11pt]{amsart}

\usepackage[left=3cm,right=3cm,top=2.5cm,bottom=3cm]{geometry}

\usepackage{amsmath,amsthm,amssymb,graphicx,color}
\usepackage{geometry}

\usepackage{amsmath,amsthm,amssymb,color,tikz,epsfig,amsfonts,latexsym,bbm,mathrsfs,natbib,hyperref}
\usepackage{epsfig,amsfonts,latexsym,MnSymbol}

\newcommand{\mbbz}{\mathbb{Z}}

\newcommand{\cA}{\mathcal{A}}
\newcommand{\beq}{\begin{equation}}
\newcommand{\eeq}{\end{equation}}
\newcommand{\alns}[1]{\begin{align*}#1\end{align*}}
\newcommand{\aln}[1]{\begin{align} #1 \end{align}}
\newcommand{\been}{\begin{enumerate}}
\newcommand{\een}{\end{enumerate}}

\newcommand{\bbn}{\mathbb{N}}
\DeclareMathOperator{\prob}{\mathbb{P}}

\DeclareMathOperator{\exptn}{\mathbb{E}}

\DeclareMathOperator{\card}{Card}

\newcommand{\eqd}{\stackrel{d}{=}}

\usepackage{cleveref}
\newtheorem{thm}{Theorem}[section]
\newtheorem{propn}[thm]{Proposition}
\newtheorem{lemma}[thm]{Lemma}
\theoremstyle{remark}

\theoremstyle{definition}

\title{Branching processes with immigration\\
in atypical random environment}

\author{Sergey Foss}
\address{Actuarial Mathematics and Statistics, Heriot-Watt University, Edinburgh EH14 4AS, Scotland, UK
and MCA, Novosibirsk State University, Russia.}
\email{S.Foss@hw.ac.uk}

\author{Dmitry Korshunov}
\address{Lancaster University, UK and Novosibirsk State University, Russia.}
\email{d.korshunov@lancaster.ac.uk}

\author{Zbigniew Palmowski}
\address{Faculty of Pure and Applied Mathematics,
Wroc\l aw University of Science and Technology,
Wyb. Wyspia\'nskiego 27, 50-370 Wroc\l aw, Poland}
\email{zbigniew.palmowski@gmail.com}

\thanks{The work of S Foss and D Korshunov is partially supported
by the RFBR grant 19-51-53010 (2019-2020).
The work of Z Palmowski is partially supported by the Polish National
Science Centre under the grant 2018/29/B/ST1/00756 (2019-2022).}

\date{\today}

\begin{document}

\begin{abstract}
Motivated by a seminal paper of \cite{kesten:kozlov:spitzer:1975}
we consider a branching process with a geometric offspring distribution
with i.i.d.\ random environmental parameters $A_n$, $n\ge 1$ and size-1 immigration 
in each generation.
In contrast to above mentioned paper we assume that the environment is long-tailed,
that is that the distribution $F$ of $\xi_n:=\log((1-A_n)/A_n)$ is long-tailed.
We prove that although the offspring distribution is light-tailed,
the environment itself can produce extremely heavy tails of the  distribution
of the population size in the $n$th generation which becomes even heavier
with increase of $n$.
More precisely, we prove that, for all $n$, the  distribution tail
$\prob(Z_n \ge m)$ of the $n$th population size $Z_n$ is asymptotically
equivalent to $n\overline F(\log m)$ as $m$ grows.
In this way we generalise \cite{BP} who proved this result in the case $n=1$
for regularly varying environment $F$ with parameter $\alpha>1$.

Further, for a subcritical branching process with subexponentially
distributed $\xi_n$, we provide asymptotics for the distribution
tail $\prob(Z_n>m)$ which are valid uniformly for all $n$, and also for the
stationary tail distribution. Then we establish the
``principle of a single atypical environment'' which says that the main
cause for the number of particles to be large is a presence
of a single very small environmental parameter $A_k$.
\vspace{3mm}

\end{abstract}

\maketitle

\pagestyle{myheadings}
\markboth{\sc  S.\ Foss --- D.\ Korshunov --- Z.\ Palmowski}
{\sc Branching process in atypical environment}

\section{Introduction and main results}

Branching processes considered in this paper are motivated by works of
\cite{solomon:1975} and \cite{kesten:kozlov:spitzer:1975},
who analysed a neighbourhood random walk in random environment.
This is a random walk $(X_t, t\in\mbbz^+)$ on $\mbbz$ defined in the following way.
Consider a collection $(A_i, i \in \mbbz^+)$ of i.i.d. $(0,1)$-valued random
variables. Let $\cA$ be the $\sigma$-algebra generated by $(A_i,i\in\mbbz^+)$.
Let $(X_k, k \in \bbn)$ be a random walk in random environment,
that is a collection of $\mbbz$-valued random variables such that $X_0=0$ and, for $k\ge 0$, 
\alns{
\prob \bigl( X_{k+1} = X_k + 1 \mid \cA, X_0 = i_0, \ldots, X_k = i_k \bigr) = A_{i_k}
}
and
\alns{
\prob \bigl( X_{k+1} = X_k -1 \mid \cA, X_0 = i_0, \ldots, X_k = i_k\bigr) = 1- A_{i_k}
}
for all $i_j \in \mbbz$, $0\le j\le k$.
The collection $(A_i,i\in\mbbz^+)$ is called a random environment.

For this random walk, \cite{kesten:kozlov:spitzer:1975}
studied the appropriately scaled limiting distribution of the hitting time
$T_n=\inf\{k > 0: X_k =n\}$ of any state state $n\in\mbbz$. 
Their analysis is based on the representation of $T_n$, $n>0$ in terms of the total 
number of particles up to the $n$th generation of a certain branching
process in random environment with size-1 immigration at each
generation step. In this model the offspring distribution in the $n$th
generation is geometric with a random parameter $A_n$.

In other words, let $(Z_n, n \ge 0)$ be a branching process in random environment
with one immigrant each time that starts from $Z_0 \equiv 0$.
Then the following representation holds:
\aln{
Z_{n+1} = \sum_{i=1}^{Z_n + 1} B_{n+1,i}\label{Zn}
}
where, conditioned on $\cA$,  $(B_{n+1,i},i\ge 1)$ are independent copies of
a geometric random variable $B_{n+1}$ with probability mass function
\aln{
\prob (B_{n+1}=k) = A_n(1-A_n)^k
\quad \mbox{ for all } k \ge 0,\ n\ge 0. \label{eq:prob:mass:func:Bn}
}
Following \cite{kesten:kozlov:spitzer:1975},
let $U_i^n$ denote the number of transitions of $(X_k,k \ge 0)$ from $i$ to $i-1$
within time interval $[0,T_n)$, i.e.,
\alns{
U_i^n = \card \{ k < T_n : X_k = i, X_{k+1}= i-1\},
}
where $\card(C)$ is the cardinality of the set $C$. It is easy to derive that
\begin{equation}
T_n = n + 2 \sum_{i = -\infty}^{\infty} U_i^n.\label{mainrepr}
\end{equation}
Note that $U_i^n = 0$ for all $i \ge n$ and $U:=\sum_{i \le 0} U_i^n < \infty$
a.s. if $X_k \to \infty$ a.s. as $k \to \infty$.
It has been established in \cite{kesten:kozlov:spitzer:1975}, that
\begin{equation}
\sum_{i=1}^n U_i^n\ \eqd\ \sum_{l=0}^{n-1} Z_l.\label{disteq}
\end{equation}
Then \cite{kesten:kozlov:spitzer:1975}
have analysed $T_n$ under the so-called ``Kesten assumptions'' on the environment: 
\aln{
\exptn \biggl( \log \frac{1-A}{A} \biggr) < 0\quad \mbox{ but }\
\exptn \biggl(\frac{1-A}{A} \biggr) \ge 1 \label{eq:assumption:solomon}
}
and there exists a unique positive solution $\kappa$ to the equation
\aln{
\exptn \bigg(\biggl(\frac{1-A}{A}\bigg)^\kappa\biggr)
\ =\ \exptn \bigg( \exp \bigg\{ \kappa \log \frac{1-A}{A} \bigg\} \bigg)
\ =\ 1. \label{eq:assumption:kks}
}
In particular, the assumption \eqref{eq:assumption:kks} implies that
the random variable
$$
\xi:=\log\frac{1-A}{A}
$$
has an exponentially decaying right tail.
It was shown in \cite{kesten:kozlov:spitzer:1975} that, under the assumptions 
\eqref{eq:assumption:solomon}--\eqref{eq:assumption:kks},
the distributions of appropriately scaled random variables $T_n$ and $\sum_{k=0}^{n-1} Z_k$ become close to each other and converge, as $n\to\infty$, to the distribution of 
a $\kappa$-stable random variable.

The tail asymptotics for the branching process $Z_n$ under the assumptions 
\eqref{eq:assumption:solomon}--\eqref{eq:assumption:kks} were studied by
\cite{DS2017} for all three regimes, subcritical,
critical, and supercritical. 

The aim of our paper is to study the asymptotic behaviour of the branching
process $Z_n$ under the complementary assumption that the distribution $F$ of
the random variable $\xi$ is {\it long-tailed}, that is, $\overline{F}(x)>0$
for all $x$ and
\begin{eqnarray}\label{assregvar}
\overline F(x-y) &\sim& \overline F(x) \quad\mbox{as }x\to\infty,
\end{eqnarray}
for some (and therefore for all) fixed $y\not=0$. Here $\overline F(x)=1-F(x)$
is the tail distribution function and equivalence \eqref{assregvar} means that
the ratio of the left- and right-hand sides tends to 1 as $x$ grows.
In particular, \eqref{assregvar} implies that $F$ is {\it heavy-tailed}, 
i.e. ${\mathbb E} e^{c\xi}=\infty$ for all $c>0$.
Given \eqref{assregvar}, the distribution $G$ defined by its tail as
$\overline G(x)=\overline F(\log x)$, $x\ge 1$, 
is {\it slowly varying at infinity} and therefore {\it subexponential}, that is,
\begin{eqnarray}\label{asssubexp}
\overline{G*G}(x) &\sim& 2\overline G(x)
\quad\mbox{as }x\to\infty,
\end{eqnarray}
see, e.g. Theorem 3.29 in \cite{FKZ}.

A distribution $F$ with finite mean is called {\it strong subexponential} if
\begin{eqnarray}\label{asssubexp.str}
\int_0^x \overline F(x-y)\overline F(y)dy &\sim&
2\overline F(x)\int_0^\infty\overline F(y)dy
\quad\mbox{as }x\to\infty.
\end{eqnarray}
Any strong subexponential distribution $F$ is subexponential, and  its
{\it integrated tail distribution} $F_I$ with the tail distribution function
\begin{eqnarray*}
\overline F_I(x) &=& \min\Bigl(1,\ \int_x^\infty\overline F(y)dy\Bigr).
\end{eqnarray*}
is subexponential too
(see e.g. \cite[Theorem 3.27]{FKZ}).
In what follows, we write $F_I(x,y]:=\overline F_I(x)-\overline F_I(y)$.

We start now with our first main result.

\begin{thm}\label{thm:tail:fixed:generation:size}
Under the assumption \eqref{assregvar},
\alns{
\prob(Z_1> m)\ \sim\ \overline F(\log m)\quad\mbox{as }m\to\infty.
}
If, in addition, the distribution $F$ is subexponential, then, for any fixed $n\ge 2$,
\alns{
\prob(Z_n> m)\ \sim\ n\overline F(\log m)\quad\mbox{as }m\to\infty.
}
\end{thm}

Theorem \ref{thm:tail:fixed:generation:size} shows that the tail of
$Z_1$ is surprisingly heavy and is getting heavier in each next  generation.
What should be underlined, this type of behaviour is a consequence of the environment only, and not of the
branching mechanism which is of geometric type.
In contrast to a  series of papers
\cite{seneta}, \cite{darling}, \cite{schuh:barbour}, \cite{hong:zhang:2018},  
we do not analyse the convergence results for $n\to\infty$, 
with focusing on the tail behaviour of the distribution of $Z_n$ for each $n$.

Consider now a branching process with state-independent immigration satisfying the stability condition
\begin{eqnarray}\label{stab.cond}
-a\ :=\ \exptn \xi < 0\quad \mbox{where }\exptn |\xi|<\infty.
\end{eqnarray}
The classical Foster criterion implies that the distribution of $Z_n$
stabilises in time, i.e.\ the distribution of the Markov chain $Z_n$ 
converges to a unique limiting/stationary distribution as $n$ grows.
It follows from Theorem \ref{thm:tail:fixed:generation:size}
that, for any $n$, the tail of the stationary distribution must be
asymptotically heavier than $n\overline{F}(\log m)$,
i.e. $\prob(Z>m)/\overline{F}(\log m)\to\infty$ as $m\to\infty$,
where $Z$ is sampled from the stationary distribution.
The distribution tail asymptotics of $Z_n$ and $Z$ are specified in the following two results.
The first result provides two asymptotic lower bounds, for finite and infinite time horizons, 
where the first bound is uniform for all generations.

\begin{thm}\label{thm:tail:generation:size.lower}
Assume that $A\le\widehat A$  a.s. for some constant $\widehat A<1$.
Then the following lower bounds hold.

{\rm (i)} If the distribution $F$ is long-tailed, then
\begin{eqnarray}\label{eq:stat1}
\prob(Z_n>m) &\ge& (a^{-1}+o(1)) F_I(\log m,\log m+na]
\mbox{ as }m\to\infty\mbox{ uniformly for all }n\ge 1.
\end{eqnarray}

{\rm (ii)} If the integrated tail distribution $F_I$ is long-tailed and
the stability condition \eqref{stab.cond} holds, then
\begin{eqnarray}\label{eq:stat2}
\prob(Z>m) &\ge& (a^{-1}+o(1)) \overline F_I(\log m)\quad\mbox{as }m\to\infty.
\end{eqnarray}
\end{thm}

The next result presents conditions for existence of upper bounds that match the lower bounds of Theorem \ref{thm:tail:generation:size.lower}.

\begin{thm}\label{thm:tail:generation:size.upper}
Let the stability condition \eqref{stab.cond} hold
and the distribution $F$ be such that
\begin{eqnarray}\label{cond.sqrt}
\overline F(m-\sqrt m)\sim\overline F(m)\quad
\mbox{ and }\quad\overline F(m)e^{\sqrt m}\to\infty\quad\mbox{as }m\to\infty.
\end{eqnarray}
Then the following upper bounds hold.

{\rm (i)} If the distribution $F$ is strong subexponential, then
\begin{eqnarray}\label{eq:stat1.eq}
\prob(Z_n>m) &\le& (a^{-1}+o(1)) F_I(\log m,\log m+na]
\mbox{ as }m\to\infty
\mbox{ uniformly for all }n\ge 1.
\end{eqnarray}

{\rm (ii)} If the integrated tail distribution $F_I$ is subexponential, then
\begin{eqnarray}\label{eq:stat2.eq}
\prob(Z>m) &\le& (a^{-1}+o(1)) \overline F_I(\log m)\quad\mbox{as }m\to\infty.
\end{eqnarray}
\end{thm}

Distributions satisfying the first condition in \eqref{cond.sqrt} are
called {\it square-root insensitive}, see e.g. \cite[Sect. 2.8]{FKZ}.
Typical examples of distributions satisfying \eqref{cond.sqrt} are:
any regularly varying distribution, the log-normal distribution and  
a Weibull distribution with parameter less than $1/2$.

We do not know, how essential is the square-root insensitivity condition for the upper bounds in Theorem \ref{thm:tail:generation:size.upper} to hold. 
In the literature, 
there are various scenarios where extra randomness leads to appearance 
of further terms in the tail asymptotics due to the effects
caused by the central limit theorem. Namely, for the Weibull distribution 
$\overline{F}(x) =\exp (-x^{\beta})$ with parameter $\beta\in[1/2,1)$,  
the number of extra terms appearing in the tail asymptotics depends on the interval $[n/(n+1), (n+1)/(n+2))$, 
$n=1$, $2$, \ldots\ the parameter $\beta$ belongs to -- 
see e.g. \cite{AKS1998} and \cite{FK2000} 
for the distributional tail asymptotics of the stationary queue length in a single-server queue
or \cite{DKW2020} for the tail asymptotics of the stationary distribution in a 
Markov chain with asymptotically zero drift. 
However, we are not certain that similar arguments may be relevant to the
model considered in the present paper.

If the distribution $F$ satisfies all the conditions of Theorems
\ref{thm:tail:generation:size.lower} and
\ref{thm:tail:generation:size.upper}, then the corresponding lower and upper 
bounds match each other and we conclude the following tail asymptotics:
\begin{eqnarray}\label{eq:stat1.eq.asy}
\prob(Z_n>m) &\sim& a^{-1} F_I(\log m,\log m+na]
\mbox{ as }m\to\infty\mbox{ uniformly for all }n\ge 1,\\
\label{eq:stat2.eq.asy}
\prob(Z>m) &\sim& a^{-1} \overline F_I(\log m)\quad\mbox{as }m\to\infty.
\end{eqnarray}


These asymptotics may be intuitively interpreted as follows:
$Z_n$ takes a large value if one of the $\xi$'s is sufficiently large, 
i.e.\ one of the success probabilities $A$'s is small.
This phenomenon may be named as {\it the principle of a single atypical 
environment} and formulated as follows.

For any $c>1$ and $\varepsilon>0$ let us introduce events
\begin{eqnarray*}
E_n^{(k)}(m) &=&
\bigl\{Z_k\le c,\ \xi_k>\log m+(a+\varepsilon)(n-k),\\
&&\hspace{15mm}|S_{j,n-1}-(n-j)\exptn\xi|\le c+\varepsilon(n-j)
\mbox{ for all }j\in[k+1,n-1]\bigr\},\quad k\le n-1,
\end{eqnarray*}
where $S_{j,n}:=\xi_j+\ldots+\xi_n$.
The event $E_n^{(k)}(m)$ describes all trajectories such that the value of $Z_k$ is relatively small, then 
the success probability $A_k$ is close to zero and, 
as a result, a single atypical environment occurs, and after time $k$ 
the environment follows the strong law of large numbers with drift $-a$.
As stated in the next theorem, the union of all these events
provides the most probable way for  the large deviations of $Z_n$ to occur.

\begin{thm}\label{thm:PSLE}
Assume that conditions of Theorems \ref{thm:tail:generation:size.lower} and
\ref{thm:tail:generation:size.upper} hold. Then, for any fixed $\varepsilon>0$,
\begin{eqnarray}\label{eq:PSLE}
\lim_{c\to\infty}\lim_{m\to\infty}
\inf_{n\ge 1}\prob\biggl(\bigcup_{k=0}^{n-1} E_n^{(k)}(m)\ \Big|\ Z_n>m\biggr)
&=& 1.
\end{eqnarray}
\end{thm}

Let us highlight a natural link of branching processes in the random
environment to stochastic difference equations.
It follows from the recurrence equation 
\begin{eqnarray*}
\exptn(Z_n\mid\mathcal A,\ Z_{n-1}) &=& (Z_{n-1}+1)\exptn(B_n\mid\mathcal A)\\
&=& (Z_{n-1}+1) \bigg(\frac{1}{A_{n-1}}-1\biggr)
\ =\ (Z_{n-1}+1)e^{\xi_{n-1}}
\end{eqnarray*}
that, for each $n$, the conditional expectation of $Z_n$,
\begin{eqnarray}\label{perp}
\exptn(Z_n\mid\mathcal A) &=&
\sum_{k=0}^{n-1} e^{\sum_{l=k}^{n-1} \xi_l}
\ =\ \sum_{k=0}^{n-1}e^{S_{k,n-1}},
\end{eqnarray}
is distributed as a finite time horizon perpetuity, and its limit $\exptn(Z\mid\mathcal A)$
as the solution to the stochastic fixed point equation. Their tail asymptotic behaviour
in the heavy-tailed case is the same as given in
\eqref{eq:stat1.eq.asy}--\eqref{eq:stat2.eq.asy}, that is,
\begin{eqnarray}\label{perp1}
\prob\bigl[\exptn(Z_n\mid\mathcal A)>m\bigr] &\sim& a^{-1} F_I(\log m,\log m+na]
\mbox{ as }m\to\infty\mbox{ uniformly for all }n\ge 1,\\
\label{perp2}
\prob\bigl[\exptn(Z\mid\mathcal A)>m\bigr] &\sim& a^{-1} \overline F_I(\log m)\quad\mbox{as }m\to\infty,
\end{eqnarray}
see \cite{Dyszewski} for \eqref{perp2} and \cite{Dima2020} for general case.

The remainder of the paper is dedicated to the proofs of the results above.
We close our paper by Section \ref{sec:extensions}
which contains some discussion and possible extensions.

\section{Finite time horizon tail asymptotics,
proof of Theorem~\ref{thm:tail:fixed:generation:size}}
\label{sec:n}

We start with some useful representations. Firstly,
\begin{eqnarray}\label{rep.Z.m.EA}
\prob(Z_1>m) &=& \prob(B_1>m)\ =\ \exptn \bigl((1-A_0)^{m+1}\bigr).
\end{eqnarray}

Secondly let us observe that the $k$-fold convolution of geometric distribution
is known in the closed form, and its probability mass function is hypergeometric:
\begin{eqnarray*}
\prob(B_1+\ldots+B_k=m\mid\mathcal{A}) &=& A^k(1-A)^m\frac{(m+1)\ldots(m+k-1)}{(k-1)!}
\quad\mbox{for all }k\ge 2\mbox{ and }m\ge 0.
\end{eqnarray*}
Therefore, for $k\ge 2$,
\begin{eqnarray*}
\prob(B_1+\ldots+B_k>m\mid\mathcal{A}) &=& (-1)^{k-1}\frac{A^k}{(k-1)!}
\frac{{\rm d}^{k-1}}{{\rm d}A^{k-1}}\sum_{n=m+1}^\infty(1-A)^{n+k-1}\\
&=& (-1)^{k-1}\frac{A^k}{(k-1)!}
\frac{{\rm d}^{k-1}}{{\rm d}A^{k-1}}\frac{(1-A)^{m+k}}{A}\\
&=& (-1)^{k-1}\frac{A^k}{(k-1)!}
\sum_{j=0}^{k-1} {{k-1}\choose{j}}
\frac{{\rm d}^j}{{\rm d}A^j}(1-A)^{m+k}
\frac{{\rm d}^{k-1-j}}{{\rm d}A^{k-1-j}}\frac{1}{A},
\end{eqnarray*}
which yields the following binomial representation that is convenient for further analysis,
\begin{eqnarray}\label{reprenegbinomial}
\prob(B_1+\ldots+B_k>m\mid\mathcal{A}) &=& A^k
\sum_{j=0}^{k-1} {{m+k}\choose{j}}
(1-A)^{m+k-j} \frac{1}{A^{k-j}}\nonumber\\
&=& \sum_{j=0}^{k-1} {{m+k}\choose{j}} A^j(1-A)^{m+k-j}.\label{reprenegbinomial}
\end{eqnarray}

The above representations allow us to prove 
two auxiliary results.

\begin{lemma}\label{l:1.exp} Under the assumption \eqref{assregvar},
\begin{eqnarray}\label{rep.Z.m.EA}
\exptn \bigl((1-A)^m\bigr) &\sim& \overline F(\log m)
\quad\mbox{as }m\to\infty.
\end{eqnarray}
\end{lemma}

\begin{lemma}\label{l:jm.upper}
Under the assumption \eqref{assregvar},
there exist $\gamma<\infty$ and $\varepsilon>0$ such that
\begin{eqnarray*}
\exptn A^j(1-A)^m &\le&
\gamma\frac{j^jm^m}{(m+j)^{m+j}}\overline F(\log m-\log j)
\quad\mbox{for all }m>1\mbox{ and }j\le\varepsilon m.
\end{eqnarray*}
In particular, for any fixed $j\ge 1$,
\begin{eqnarray}
\exptn A^j(1-A)^m &=& o(\overline F(\log m))\quad\mbox{as }m\to\infty. \label{osmalljgeq1}
\end{eqnarray}
\end{lemma}

\begin{proof}[Proof of Lemma \ref{l:1.exp}]
Since, for any fixed $\varepsilon>0$,
$$
\exptn \bigl((1-A)^{m+1};\ A>\varepsilon\bigr)\ \le\ (1-\varepsilon)^{m+1}
$$
is exponentially decreasing as $m\to\infty$, the asymptotic behaviour of the right-hand side
in \eqref{rep.Z.m.EA} is determined by
the tail behavior of $A$ near $0$. Notice that, for $0<a<b<1$,
\begin{eqnarray}\label{startid}
\prob\left(A\in(a, b]\right) &=&
\prob\left(\log \frac{1-A}{A}\in \left[\log \frac{1-b}{b},\ \log \frac{1-a}{a}\right)\right)
\nonumber\\
&=& \prob\bigl(\xi\in [\log(1/b-1),\ \log(1/a-1))\bigr).
\end{eqnarray}
Hence, for any fixed $c>0$, we have
\begin{eqnarray*}
\exptn (1-A)^m &\ge& \exptn [(1-A)^m;\ A\le c/m] \\
&\ge& (1-c/m)^m\prob( A\le c/m)\\
&=& (1-c/m)^m\overline F(\log(m/c-1)).
\end{eqnarray*}
It follows from the long-tailedness of the distribution $F$ of $\xi$ that
the right-hand side of above equation is asymptotically equivalent to
$e^{-c}\overline F(\log m)$ as $m\to\infty$.
Letting $c\downarrow 0$ we complete the proof of the lower bound
\begin{eqnarray*}
\exptn (1-A)^m &\ge& (1+o(1))\overline F(\log m)\quad\mbox{as }m\to\infty.
\end{eqnarray*}

To obtain the matching upper bound, let us consider the following decomposition
which is valid for all integer $K\in[1,[m/2]-1]$:
\begin{eqnarray*}
\lefteqn{\exptn (1-A)^m}\\
 &=& \exptn\biggl[(1-A)^m;\ A\le \frac{K}{m}\biggr]+
\sum_{k=K}^{[m/2]-1} \exptn \biggl[(1-A)^m;\
A\in\biggl(\frac{k}{m},\frac{k+1}{m}\biggr]\biggr]
+\exptn\biggl[(1-A)^m;\ A>\frac{[m/2]}{m}\biggr]\\
&\le& \prob\biggl(A\le \frac{K}{m}\biggr) 
+\sum_{k=K}^{[m/2]-1}\biggl(1-\frac{k}{m}\biggr)^m 
\prob\biggl(A\le\frac{k+1}{m}\biggr)
+\biggl(1-\frac{[m/2]}{m}\biggr)^m\\
&\le& \overline F\biggl(\log\biggl(\frac{m}{K}-1\biggr)\biggr)
+\sum_{k=K}^{[m/2]-1} e^{-k} \overline F\biggl(\log\biggl(\frac{m}{k+1}-1\biggr)\biggr)
+\biggl(1-\frac{[m/2]}{m}\biggr)^m.
\end{eqnarray*}
Let us show that the series in the middle term in the last line is
negligible for large values of $K$. Indeed, firstly,
$$
\frac{m}{k+1}-1\ \ge\ \frac{1}{2}\frac{m}{k+1}
\quad\mbox{for all }k\le \frac{m}{2}-1
$$
and hence
\begin{eqnarray*}
\sum_{k=K}^{[m/2]-1} e^{-k} 
\overline F\biggl(\log\biggl(\frac{m}{k+1}-1\biggr)\biggr) &\le&
\sum_{k=K}^{[m/2]-1} e^{-k} \overline F(\log m-\log(k+1)-\log 2).
\end{eqnarray*}
Since the distribution $F$ is assumed long-tailed, there exists
a constant $\gamma<\infty$ such that
$\overline F(x-y)\le\gamma e^y \overline F(x)$ for all $x$, $y>0$. Therefore,
\begin{eqnarray}\label{inner.sum.1}
\sum_{k=K}^{[m/2]-1} e^{-k} 
\overline F\biggl(\log\biggl(\frac{m}{k+1}-1\biggr)\biggr) &\le&
\gamma\overline F(\log m)\sum_{k=K}^\infty e^{-k}
e^{\log(k+1)+\log 2}\nonumber\\
&\le& \varepsilon(K)\overline F(\log m)
\end{eqnarray}
where
\begin{eqnarray*}
\varepsilon(K) &:=&
\gamma\sum_{k=K}^\infty e^{-k} e^{\log(k+1)+\log 2}\ \to\ 0
\quad\mbox{as }K\to\infty.
\end{eqnarray*}
Hence we conclude that
\begin{eqnarray*}
\exptn (1-A)^m &\le& \overline F(\log(m/K-1))
+\varepsilon(K)\overline F(\log m)+O(1/2^m)
\quad\mbox{as }m\to\infty.
\end{eqnarray*}
Due to the long-tailedness of $F$ this implies that, for any fixed $K$,
\begin{eqnarray*}
\exptn (1-A)^m &\le&
(1+o(1))\overline F(\log m)+\varepsilon(K)\overline F(\log m)
\quad\mbox{as }m\to\infty.
\end{eqnarray*}
Since $\varepsilon(K)\to 0$ as $K\to\infty$, the proof is complete.
\end{proof}

\begin{proof}[Proof of Lemma \ref{l:jm.upper}]
There exist $K\in\mathbb N$ and $\varepsilon_1>0$
such that the following inequalities hold
\begin{eqnarray}\label{log.lin}
\log(k+1) &\le& k/6\quad\mbox{for all }k\ge K
\end{eqnarray}
and
\begin{eqnarray}\label{comp.e}
\biggl(1-\frac{j}{m}\biggr)^m &\ge& \frac{1}{3^j}
\quad\mbox{for all }m>K\mbox{ and }j\le\varepsilon_1 m.
\end{eqnarray}
Similar to the case $j=0$ considered in the proof of Lemma \ref{l:1.exp},
we make use of the following decomposition:
\begin{eqnarray}\label{E1.E2}
\exptn A^j(1-A)^m &=&
\exptn\biggl[A^j(1-A)^m;\ A\le\frac{Kj}{3m}\biggr]+
\sum_{k=K}^{[3m/j]} \exptn\biggl[A^j(1-A)^m;\
A\in\biggl(k\frac{j}{3m},(k+1)\frac{j}{3m}\biggr]\biggr]\nonumber\\
&=:& E_1+E_2.
\end{eqnarray}
The maximum of the function $x^j(1-x)^m$ over the interval $[0,1]$
is attained at point $j/(m+j)$ and is equal to $j^jm^m/(m+j)^{m+j}$.
Therefore, for some $\varepsilon=\varepsilon(K)\le\varepsilon_1$,
\begin{eqnarray}\label{E1}
E_1 &\le& \frac{j^jm^m}{(m+j)^{m+j}} \prob\biggl(A\le\frac{Kj}{3m}\biggr)\nonumber\\
&=& \frac{j^jm^m}{(m+j)^{m+j}}
\overline F\biggl(\log\biggl(\frac{3m}{Kj}-1\biggl)\biggl)\nonumber\\
&\le& \gamma_1\frac{j^jm^m}{(m+j)^{m+j}} \overline F(\log m-\log j)
\quad\mbox{for some }\gamma_1<\infty\mbox{ and all }j\le\varepsilon m,
\end{eqnarray}
owing to the long-tailedness of $F$. Further, the series on the right
hand side of \eqref{E1.E2} possesses the following upper bound
\begin{eqnarray*}
E_2 &\le&
\sum_{k=K}^{[3m/j]} (k+1)^j\biggl(\frac{j}{3m}\biggr)^j
\biggl(1-\frac{kj}{3m}\biggr)^m \prob\biggl(A\le(k+1)\frac{j}{3m}\biggr)\\
&\le& \biggl(\frac{j}{3m}\biggr)^j\
\sum_{k=K}^{[3m/j]} (k+1)^j e^{-kj/3}
\overline F(\log(3m/(k+1)j-1))
\end{eqnarray*}
because $(1-kj/3m)^m\le e^{-kj/3}$.
Let us now bound the latter series.
It follows from the inequality \eqref{log.lin} that
\begin{eqnarray*}
(k+1)^j e^{-kj/3} &=& e^{j(\log(k+1)-k/3)}
\ \le\ e^{-jk/6}\quad\mbox{for all }k\ge K.
\end{eqnarray*}
Then, using arguments similar to those in \eqref{inner.sum.1},
\begin{eqnarray}
E_2 &\le& \biggl(\frac{j}{3m}\biggr)^j
\sum_{k=K}^{[3m/j]} e^{-jk/6} \overline F(\log(3m/(k+1)j-1))\label{citedlater2}\\
&\le& \gamma_2\biggl(\frac{j}{3m}\biggr)^j\overline F(\log m-\log j)
\quad\mbox{for some }\gamma_2<\infty,\nonumber
\end{eqnarray}
which implies the result due to the inequalities \eqref{E1} and
\begin{eqnarray*}
\frac{j^jm^m}{(m+j)^{m+j}}\ =\
\biggl(\frac{j}{m}\biggr)^j\biggl(1-\frac{j}{m+j}\biggr)^{m+j}
&\ge& \biggl(\frac{j}{3m}\biggr)^j
\end{eqnarray*}
which is guarantied by \eqref{comp.e}.
\end{proof}

\begin{proof}[Proof of Theorem~\ref{thm:tail:fixed:generation:size}]
We prove the statement by induction in $n\ge 1$.
The assertion for $n=1$ follows from the representation \eqref{rep.Z.m.EA}
and Lemma \ref{l:1.exp}.
Assume that the assertion of Theorem~\ref{thm:tail:fixed:generation:size}
is valid for some $n\ge 1$. Let us show that then it follows for $n+1\ge 2$.
Our aim is to obtain the tail asymptotics of the distribution of
\alns{
Z_{n+1}\ =\ \sum_{i=1}^{Z_n+1} B_{n+1,i},
}
where $(B_{n+1,i}, i\ge 1)$ are independent copies of a geometric random
variable $B_{n+1}$ with success probability $A_n$ (its probability mass function
is specified in \eqref{eq:prob:mass:func:Bn}) and independent of $Z_n$
conditioned on $\mathcal{A}$. Then the following representation holds
\begin{eqnarray}\label{decomp}
\prob(Z_{n+1}>m) &=& \sum_{k=0}^\infty
\prob \bigg( \sum_{j=1}^{k+1} B_{n+1,j}>m, Z_n=k\bigg)\nonumber\\
&=& \sum_{k=0}^\infty
\exptn \bigg[ \prob \bigg(\sum_{j=1}^{k+1} B_{n+1,j}>m\Big|
\mathcal{A} \bigg) \bigg] \prob(Z_n=k),
\end{eqnarray}
where we have conditioned on $\mathcal{A}$ and used the fact that $Z_n$ and
$(B_{n+1,i}, i\ge 1)$ are independent conditioned on $\mathcal{A}$.

We start with the proof of the upper bound.
For that, let us split the summation in \eqref{decomp} into three parts,
from $0$ to $K$, from $K+1$ to $\varepsilon m-1$ and from $\varepsilon m$ to $\infty$
where integer $K$ is chosen large enough and real $\varepsilon>0$ small enough.
This splitting together with non-negativity of the $B$'s implies that
\begin{eqnarray*}
\lefteqn{\prob(Z_{n+1}>m)} \\
&\le& \exptn \bigg[ \prob \bigg( \sum_{j=1}^K B_{n+1,j}>m\Big|
\mathcal{A} \bigg) \bigg] \prob(Z_n<K)\\
&&\hspace{42mm}+ \sum_{k=K}^{\varepsilon m}
\exptn \bigg[ \prob \bigg( \sum_{j=1}^{k+1} B_{n+1,j}>m\Big|
\mathcal{A} \bigg) \bigg] \prob(Z_n=k) + \prob(Z_n>\varepsilon m)\\
&\le& \exptn \bigg[ \prob \bigg( \sum_{j=1}^K B_{n+1,j}>m\Big|
\mathcal{A} \bigg)\bigg]
+ \sum_{k=K}^{\varepsilon m}\exptn
\bigg[ \prob \bigg( \sum_{j=1}^{k+1} B_{n+1,j}>m\Big|
\mathcal{A} \bigg) \bigg] \prob(Z_n=k) + \prob(Z_n>\varepsilon m).
\end{eqnarray*}
By the induction hypothesis and long-tailedness of $F$,
for any fixed $\varepsilon$,
\begin{eqnarray*}
\prob(Z_n>\varepsilon m) &\sim& n\overline F(\log(\varepsilon m))
\ \sim\ n\overline F(\log m)\quad\mbox{as }m\to\infty.
\end{eqnarray*}
So it is left to show that, for any fixed $K$,
\begin{eqnarray}\label{est.upper.44}
\exptn \bigg[ \prob \bigg( \sum_{j=1}^K B_{n+1,j}>m\Big|
\mathcal{A} \bigg)\bigg]
&\sim& \overline F(\log m)\quad\mbox{as }m\to\infty,
\end{eqnarray}
and that, for any $\delta>0$, there exist a sufficiently large $K$
and a sufficiently small $\varepsilon>0$ such that
\begin{eqnarray}\label{est.upper.4}
\sum_{k=K}^{\varepsilon m}\exptn \bigg[ \prob \bigg( \sum_{j=1}^{k+1} B_{n+1,j}>m\Big|
\mathcal{A} \bigg) \bigg] \prob(Z_n=k) &\le& \delta\overline F(\log m)
\quad\mbox{for all sufficiently large }m.
\end{eqnarray}
We start with proving \eqref{est.upper.4}.

Let $\xi(A)$ be a Bernoulli random variable with success probability $A$
and $S_{m+k}(A)$ be the sum of $m+k$ independent copies of $\xi(A)$.
It follows from the representation \eqref{reprenegbinomial} that
\begin{eqnarray*}
\prob(B_1+\ldots+B_k>m\mid\mathcal{A}) &=&
\prob(S_{m+k}(A)\le k-1)\\
&\le& (\exptn (e^{-\beta\xi(A)}))^{m+k}e^{\beta(k-1)}\\
&=& (1-A+e^{-\beta}A)^{m+k}e^{\beta(k-1)}, \quad\mbox{for all }\beta>0.
\end{eqnarray*}
The minimal value of the right hand side is attained for $\beta$
such that $e^{-\beta}=\frac{(1-A)(k-1)}{A(m+1)}$, hence
\begin{eqnarray*}
\prob(B_1+\ldots+B_k>m\mid\mathcal{A}) &\le&
\frac{(m+k)^{m+k}}{(m+1)^{m+1}(k-1)^{k-1}}A^{k-1}(1-A)^{m+1}.
\end{eqnarray*}
This allows us to conclude from Lemma \ref{l:jm.upper} that,
for $k\le\varepsilon m$,
\begin{eqnarray*}
\prob(B_1+\ldots+B_k>m) &\le&
\frac{(m+k)^{m+k}}{(m+1)^{m+1}(k-1)^{k-1}}\exptn A^{k-1}(1-A)^{m+1}\\
&\le& \gamma\overline F(\log(m+1)-\log(k-1)).
\end{eqnarray*}
Therefore,
\begin{eqnarray*}
\sum_{k=K}^{\varepsilon m}\prob(B_1+\ldots+B_{k+1}>m)\prob(Z_n=k)
&\le& \gamma\sum_{k=K}^{\varepsilon m} \overline F(\log(m+1)-\log k)\prob(Z_n=k).
\end{eqnarray*}
Representing $\prob(Z_n=k)$ as the difference 
$\prob(Z_n>k-1)-\prob(Z_n>k)$ and rearranging the sum on the right hand side 
we conclude that this sum is not greater than
\begin{eqnarray*}
\lefteqn{\overline F(\log(m+1)-\log K)\prob(Z_n>K-1)}\\
&&+ \sum_{k=K}^{\varepsilon m-1} \bigl(
\overline F(\log(m+1)-\log (k+1))-\overline F(\log(m+1)-\log k)\bigr)
\prob(Z_n>k).
\end{eqnarray*}
Then the induction hypothesis yields an upper bound, for some $\gamma_1<\infty$,
\begin{eqnarray*}
\lefteqn{\sum_{k=K}^{\varepsilon m}\prob(B_1+\ldots+B_{k+1}>m)\prob(Z_n=k)}\\
&\le& \gamma\overline F(\log(m+1)-\log K)\prob(Z_n>K-1)\\
&&+ \gamma_1\sum_{k=K}^{\varepsilon m-1} \bigl(
\overline F(\log(m+1)-\log (k+1))-\overline F(\log(m+1)-\log k)\bigr)
\overline F(\log k).
\end{eqnarray*}
Due to the long-tailedness of $F$, for any $\delta>0$ there exists
a sufficiently large $K$ such that the first term on the right hand side
is not greater than $\delta\overline F(\log m)$, for all sufficiently large $m$.
After rearranging we conclude that the sum on the right hand side
is not greater than
\begin{eqnarray}\label{sum.2}
\lefteqn{\overline F(\log(m+1)-\log (\varepsilon m))
\overline F(\log(\varepsilon m-1))}\nonumber\\
&&\hspace{30mm}+\sum_{k=K+1}^{\varepsilon m-1} \overline F(\log(m+1)-\log k)
\bigl(\overline F(\log(k-1))-\overline F(\log k)\bigr).
\end{eqnarray}
Since $F$ is long-tailed, the first term here is asymptotically equivalent to
\begin{eqnarray*}
\overline F(\log(1/\varepsilon))\overline F(\log m)
\quad\mbox{as }m\to\infty,
\end{eqnarray*}
so it is not greater than $\delta\overline F(\log m)$
for all sufficiently large $m$ provided
$\overline F(\log(1/\varepsilon))\le\delta/2$.
The sum in \eqref{sum.2} equals
\begin{eqnarray*}
\sum_{k=K+1}^{\varepsilon m-1} \overline G\biggl(\frac{m+1}{k}\biggr)G(k-1,k],
\end{eqnarray*}
where the distribution $G$ is defined via its tail as
$\overline G(x)=\overline F(\log x)$, and can be 
bounded by the integral 
\begin{eqnarray*}
\int_K^{\varepsilon m} \overline G(m/z)G(dz) &=&
\prob(e^{\xi_1+\xi_2}>m;\ e^{\xi_2}\in(K,\varepsilon m])\\
&=& \prob(\xi_1+\xi_2>\log m;\ \xi_2\in(\log K,\ \log m-\log(1/\varepsilon)]).
\end{eqnarray*}
Since the distribution $F$ is assumed to be subexponential,
we can choose a sufficiently large $K$ and a sufficiently small
$\varepsilon>0$ such that the latter probability
is not greater than $\delta\overline F(\log m)$ for all sufficiently large $m$,
see \cite[Theorem 3.6]{FKZ}, which completes the proof of \eqref{est.upper.4}.

To complete the proof of the upper bound it now suffices to show \eqref{est.upper.44}.
This follows immediately from the representation \eqref{reprenegbinomial},
the asymptotics \eqref{osmalljgeq1} and Lemma \ref{l:1.exp}.

We will obtain now the matching lower bound.
For that, let us split the sum in \eqref{decomp} into two parts,
from $0$ to $cm$ and from $cm+1$ to $\infty$
where $c$ is a large number sent to infinity later on.
This splitting implies that
\begin{eqnarray}\label{est}
\prob(Z_{n+1}>m) &\ge& \sum_{k=0}^{cm}
\exptn \big[ \prob(B_{n+1}>m\mid \mathcal{A})\big]\prob(Z_n=k)
+\sum_{cm+1}^\infty
\exptn \bigg[ \prob \bigg( \sum_{j=1}^{k+1} B_{n+1,j}>m\Big|
\mathcal{A} \bigg) \bigg] \prob(Z_n=k)\nonumber\\
&\ge& \exptn \big[ \prob(B_{n+1}>m\mid \mathcal{A})\big]\prob(Z_n\le cm)
+\exptn \bigg[ \prob \bigg( \sum_{j=1}^{cm} B_{n+1,j}>m\Big|
\mathcal{A} \bigg) \bigg] \prob(Z_l>cm),
\end{eqnarray}
since all the $B$'s are non-negative. By Lemma \ref{l:1.exp},
\begin{eqnarray}\label{est.1}
\exptn \big[ \prob(B_{n+1}>m\mid \mathcal{A})\big]\prob(Z_n\le cm)
&\sim& \overline F(\log m)\quad\mbox{as }m\to\infty.
\end{eqnarray}
Further, by the law of large numbers,
\begin{eqnarray*}
\prob \bigg( \sum_{j=1}^{cm} B_{n+1,j}>m\Big|\mathcal{A} \bigg)
&\stackrel{\rm a.s.}\to& 1\quad\mbox{as }c\to\infty.
\end{eqnarray*}
Hence, the dominated convergence theorem allows us to conclude that
\begin{eqnarray}\label{est.2}
\exptn \bigg[ \prob \bigg( \sum_{j=1}^{cm} B_{n+1,j}>m\Big|
\mathcal{A} \bigg) \bigg] &\to& 1\quad\mbox{as }c\to\infty.
\end{eqnarray}
Finally, by the induction hypothesis and long-tailedness of $F$, for any fixed $c$,
\begin{eqnarray}\label{est.3}
\prob(Z_n>cm) &\sim& n\overline F(\log(cm))
\ \sim\ n\overline F(\log m)\quad\mbox{as }m\to\infty.
\end{eqnarray}
Substituting \eqref{est.1}--\eqref{est.3} into \eqref{est}
and letting $c\to\infty$ we conclude the induction step for the lower bound.
\end{proof}

\section{Proof of  the lower bound, Theorem~\ref{thm:tail:generation:size.lower}}
\label{sec:lower}

Note that, by the strong law of large numbers,
for any fixed $\varepsilon>0$,
\begin{eqnarray}\label{C_S}
\inf_{n\ge 1}\prob(C_S(c,\varepsilon,k,n)\mbox{ for all }k\le n)
&\to& 1\quad\mbox{as }c\to\infty,
\end{eqnarray}
where
\begin{eqnarray*}
C_S(c,\varepsilon,k,n) &:=& \{|S_{k,n}-(n-k+1)\exptn\xi|\le c+\varepsilon(n-k+1)\}
\end{eqnarray*}
and $S_{k,n}= \xi_k+\ldots+\xi_n$.

We show that, under the long-tailedness
condition \eqref{assregvar}, the most probable way for a big value of $Z_n$ to occur 
is due to atypical random environment
when one of the following events occurs, $k\le n-1$:
\begin{eqnarray*}
C_A(k,n) &:=& \Bigl\{A_k\le \frac{c_1}{M(m,k,n)},\
C_S(c_2,\varepsilon,j,n-1)\mbox{ for all }j\in[k+1,n-1]\Bigr\},
\end{eqnarray*}
where
\begin{eqnarray*}
M(m,k,n) &:=& m e^{\varepsilon(n-1-k)+c_2}\prod_{j=k+1}^{n-1}\frac{1}{a_{A_j}}
\ =\  m e^{\varepsilon(n-1-k)+c_2-S_{k+1,n-1}},
\end{eqnarray*}
$a_A:=\exptn \{B\mid A\}=1/A-1=e^\xi$,
$c_1$, $c_2$, $\varepsilon>0$ are fixed, $c_2$ will be sent to infinity later on,
while $c_1$ and $\varepsilon$ will be sent to $0$.
Since $A$ is bounded by $\widehat A<1$, $a_A$ is bounded away from $0$
by $1/\widehat A-1$.

Let us bound from below the probability of the union of events
$C_A(k,n)$.
We start with the following lower bound
\begin{eqnarray}\label{lower.for.union}
\prob\biggl(\bigcup_{k=0}^{n-1} C_A(k,n)\biggr)
&\ge& \sum_{k=0}^{n-1}\prob(C_A(k,n))
-\sum_{k\not=l}\prob(C_A(k,n)\cap C_A(l,n)).
\end{eqnarray}
On the event $C_S(c_2,\varepsilon,k+1,n-1)$ we have
\begin{eqnarray}\label{lln}
a(n-1-k)\ \le\ \varepsilon(n-1-k)+c_2-S_{k+1,n-1}
\ \le\ 2c_2+(2\varepsilon+a)(n-1-k)
\end{eqnarray}
and hence
\begin{eqnarray*}
\sum_{k=0}^{n-1}\prob(C_A(k,n)) &\ge& \sum_{k=0}^{n-1}\prob
\biggl(A_k\le \frac{c_1}{m e^{2c_2+(2\varepsilon+a)(n-1-k)}},\
C_S(c_2,\varepsilon,j,n-1) \mbox{ for all }j\in[k+1,n-1]\biggr)\\
&=& \sum_{k=0}^{n-1}\prob
\biggl(A_k\le \frac{c_1}{m e^{2c_2+(2\varepsilon+a)(n-1-k)}}\biggr)
\prob\bigl(C_S(c_2,\varepsilon,j,n-1) \mbox{ for all }j\in[k+1,n-1]\bigr)\\
&\ge& \prob\bigl(C_S(c_2,\varepsilon,j,n-1)\mbox{ for all }j\in[1,n-1]\bigr)
\sum_{k=0}^{n-1}
\prob\biggl(A_k\le \frac{c_1}{m e^{2c_2+(2\varepsilon+a)(n-1-k)}}\biggr),
\end{eqnarray*}
and
\begin{eqnarray*}
\sum_{k\not=l}\prob(C_A(k,n)\cap C_A(l,n)) &\le&
\sum_{k\not=l}\prob\biggl(A_k\le \frac{c_1}{m e^{a(n-1-k)}},\
A_l\le \frac{c_1}{m e^{a(n-1-l)}}\biggr)\\
&=& \sum_{k\not=l}\prob\biggl(A_k\le \frac{c_1}{m e^{a(n-1-k)}}\biggr)
\prob\biggl(A_l\le \frac{c_1}{m e^{a(n-1-l)}}\biggr)\\
&\le& \biggl(\sum_{k=0}^{n-1}
\prob\biggl(A_k\le \frac{c_1}{m e^{a(n-1-k)}}\biggr)\biggr)^2.
\end{eqnarray*}
As follows from \eqref{startid},
\begin{eqnarray*}
\sum_{k=0}^{n-1}
\prob\biggl(A_k\le \frac{c_1}{m e^{2c_2+(2\varepsilon+a)(n-1-k)}}\biggr)
&=& \sum_{k=0}^{n-1}
\prob\biggl(\xi\ge\log\biggl(\frac{m e^{2c_2+(2\varepsilon+a)k}}{c_1}
-1\biggr)\biggr)\\
&\ge& \sum_{k=0}^{n-1} \overline F(\log m +2c_2+(2\varepsilon+a)k-\log c_1)\\
&\ge& \frac{1}{2\varepsilon+a}
\int_{\log m+2c_2-\log c_1}^{\log m+2c_2-\log c_1+(2\varepsilon+a)n}
\overline F(x)dx
\end{eqnarray*}
since the tail function $\overline F(x)$ is decreasing. Therefore,
\begin{eqnarray*}
\sum_{k=0}^{n-1}
\prob\biggl(A_k\le \frac{c_1}{m e^{2c_2+(2\varepsilon+a)(n-1-k)}}\biggr)
&\ge& \frac{1+o(1)}{2\varepsilon+a}
\int_{\log m}^{\log m+(2\varepsilon+a)n} \overline F(x)dx
\end{eqnarray*}
as $m\to\infty$ uniformly for all $n\ge1$
because the distribution $F$ is long-tailed. Similarly,
\begin{eqnarray*}
\sum_{k=0}^{n-1}
\prob\biggl(A_k\le \frac{c_1}{m e^{a(n-1-k)}}\biggr)
&\le& \frac{1+o(1)}{a} \int_{\log m}^{\log m+na} \overline F(x)dx.
\end{eqnarray*}
Therefore,
\begin{eqnarray*}
\sum_{k=0}^{n-1}\prob(C_A(k,n)) &\ge& \frac{1+o(1)}{2\varepsilon+a}
\int_{\log m}^{\log m+na} \overline F(x)dx
\prob\bigl(C_S(c_2,\varepsilon,j,n-1)\mbox{ for all }j\in[1,n-1]\bigr),
\end{eqnarray*}
and
\begin{eqnarray*}
\sum_{k\not=l}\prob(C_A(k,n)\cap C_A(l,n)) &=&
O\biggl(\int_{\log m}^{\log m+na} \overline F(x)dx\biggr)^2
\end{eqnarray*}
as $m\to\infty$ uniformly for all $n\ge1$.
Substituting these bounds into \eqref{lower.for.union} and applying \eqref{C_S}, for any fixed $\varepsilon>0$
we can conclude the following lower bound,
\begin{eqnarray}\label{lower.for.union.final}
\prob\biggl(\bigcup_{k=0}^{n-1} C_A(k,n)\biggr)
&\ge& \frac{g(c_2)+o(1)}{2\varepsilon+a}
\int_{\log m}^{\log m+na} \overline F(x)dx
\end{eqnarray}
as $m\to\infty$ uniformly for all $n\ge1$,
where $g(c_2)\to 1$ as $c_2\to\infty$.

As above, conditioning on $\mathcal A$ yields
\begin{eqnarray}\label{Zn.A.C}
\prob(Z_n>m) &=& \exptn[\prob(Z_n>m\mid \mathcal A)]\nonumber\\
&\ge& \exptn[\prob(Z_n>m\mid \mathcal A);\ C_A(n)],
\end{eqnarray}
where $C_A(n):=\bigcup_{k=0}^{n-1} C_A(k,n)$.
Then, owing to \eqref{lower.for.union.final}, for the proof of \eqref{eq:stat1} it suffices to show that
\begin{eqnarray}\label{Zn.A.C.1}
\liminf_{m\to\infty}\inf_{C_A(n)}\prob(Z_n>m\mid \mathcal A) &\ge& e^{-c_1}
\quad\mbox{uniformly for all }n\ge 1.
\end{eqnarray}

Hence we are left with the proof of \eqref{Zn.A.C.1}.
Since the event $C_A(n)$ is the union of events $C_A(k,n)$, $k\le n-1$, 
the probability of the event
\begin{eqnarray*}
C_B(k,n) &:=&
\biggl\{B_{k+1,1}>m e^{c_2+\varepsilon(n-1-k)}\prod_{j=k+1}^{n-1}\frac{1}{a_{A_j}}\biggr\},
\end{eqnarray*}
conditionally on $C_A(n)$, possesses the following asymptotic lower bound
\begin{eqnarray*}
\prob(C_B(k,n)\mid C_A(n)) &\ge& 
(1-A)^{m e^{c_2+\varepsilon(n-1-k)}\prod_{j=k+1}^{n-1}\frac{1}{a_{A_j}}}\\
&\ge& \biggl(1-\frac{c_1}{m e^{c_2+\varepsilon(n-1-k)}\prod_{j=k+1}^{n-1}\frac{1}{a_{A_j}}}
\biggr)^{m e^{c_2+\varepsilon(n-1-k)}\prod_{j=k+1}^{n-1}\frac{1}{a_{A_j}}}\\
&\to& e^{-c_1}\quad\mbox{as }m\to\infty.
\end{eqnarray*}
Therefore, it only remains to show that
\begin{eqnarray}\label{l.b.ind.conv}
\inf_{C_A(k,n)}\prob(Z_n>m\mid C_B(k,n),\ \mathcal A) &\to& 1
\end{eqnarray}
as $m\to\infty$ uniformly for all $k\le n-1$ and $n\ge 1$.

To prove this convergence, let us note that, conditioned on $\mathcal A$,
\begin{eqnarray*}
\prob\bigl[Z_j\le la_{A_{j-1}}e^{-\varepsilon}
\big\mid Z_{j-1}=l,\mathcal A\bigr] &=&
\prob\bigl[B_{j,1}+\ldots+B_{j,l+1}\le la_{A_{j-1}}e^{-\varepsilon}
\big\mid \mathcal A\bigr]\\
&\le& \prob\biggl[\frac{B_{j,1}}{a_{A_{j-1}}}+\ldots+\frac{B_{j,l}}{a_{A_{j-1}}}
\le le^{-\varepsilon} \bigg\mid \mathcal A\biggr]\\
&=& \prob\biggl[\biggl(e^{-\varepsilon/2}-\frac{B_{j,1}}{a_{A_{j-1}}}\biggr)
+\ldots+\biggl(e^{-\varepsilon/2}-\frac{B_{j,l}}{a_{A_{j-1}}}\biggr)
\ge l(e^{-\varepsilon/2}-e^{-\varepsilon}) \bigg\mid \mathcal A\biggr]\\
&\le& \prob\biggl[\biggl(e^{-\varepsilon/2}-\frac{B_{j,1}}{a_{A_{j-1}}}\biggr)
+\ldots+\biggl(e^{-\varepsilon/2}-\frac{B_{j,l}}{a_{A_{j-1}}}\biggr)
\ge l e^{-\varepsilon}\varepsilon/2 \bigg\mid \mathcal A\biggr].
\end{eqnarray*}
Applying the exponential Markov inequality, we obtain the following upper bound,
for all $\lambda>0$,
\begin{eqnarray*}
\prob\bigl[Z_j\le l a_{A_{j-1}}e^{-\varepsilon}
\big\mid Z_{j-1}=l,\mathcal A\bigr] &\le&
e^{-l\lambda e^{-\varepsilon}\varepsilon/2}
\exptn e^{\lambda\bigl(\bigl(e^{-\varepsilon/2}-\frac{B_{j,1}}{a_{A_{j-1}}}\bigr)
+\ldots+\bigl(e^{-\varepsilon/2}-\frac{B_{j,l}}{a_{A_{j-1}}}\bigr)\bigr)}.
\end{eqnarray*}
Since
\begin{eqnarray*}
\exptn\Bigl[e^{\lambda\bigl(e^{-\varepsilon/2}-\frac{B}{a_A}\bigr)}
\Big\mid A\Bigr]
&=& e^{\lambda(1-\varepsilon)}\frac{A}{1-(1-A)e^{-\lambda\frac{A}{1-A}}}\\
&=& e^{\frac{\lambda}{1-A}-\lambda\varepsilon}
\frac{A}{e^{\lambda\frac{A}{1-A}}-(1-A)}
\ \le\ e^{\frac{\lambda}{1-A}-\lambda\varepsilon}
\frac{1}{\frac{\lambda}{1-A}+1}
\end{eqnarray*}
and $A$ is bounded away from $1$,
there exists a sufficiently small $\lambda_0>0$ such that
\begin{eqnarray*}
\exptn\Bigl[e^{\lambda_0\bigl(e^{-\varepsilon/2}-\frac{B}{a_A}\bigr)}
\Big\mid A\Bigr]
&\le& 1\quad\mbox{for all }A\in(0,\widehat{A}).
\end{eqnarray*}
Therefore,
\begin{eqnarray*}
\prob\bigl[Z_j\le l a_{A_{j-1}}e^{-\varepsilon}
\big\mid Z_{j-1}=l,\ \mathcal A\bigr] &\le&
e^{-l\delta}\quad\mbox{where }\delta=\lambda_0e^{-\varepsilon}\varepsilon/2>0.
\end{eqnarray*}
which, due to monotonicity property of the branching process $Z_n$, implies that
\begin{eqnarray*}
\prob\bigl[Z_j\le l a_{A_{j-1}}e^{-\varepsilon}
\big\mid Z_{j-1}\ge l,\ \mathcal A\bigr] &\le& e^{-l\delta}.
\end{eqnarray*}
Then the induction arguments lead to the following upper bound
\begin{eqnarray*}
\prob\biggl[Z_n\le l e^{-\varepsilon(n-1-k)}\prod_{i=k+1}^{n-1} a_{A_i}
\bigg\mid Z_{k+1}\ge l,\ \mathcal A\biggr] &\le&
\sum_{j=k+1}^{n-1} e^{-l\delta e^{-\varepsilon(j-1-k)}\prod_{i=k+1}^{j-1} a_{A_i}}.
\end{eqnarray*}
We take
\begin{eqnarray*}
l &=& m e^{c_2+\varepsilon(n-1-k)}\prod_{i=k+1}^{n-1}\frac{1}{a_{A_i}} 
\end{eqnarray*}
to conclude that
\begin{eqnarray*}
\prob(Z_n>m \mid C_B(k,n),\ \mathcal A) &\ge&
1-\sum_{j=k+1}^{n-1} e^{-m\delta e^{c_2+\varepsilon(n-1-j)}\prod_{i=j+1}^{n-1} a_{A_i}}.
\end{eqnarray*}
Due to the representation
\begin{eqnarray*}
\log e^{c_2}\prod_{i=j}^{n-1}\frac{A_i}{1-A_i} &=&
c_2+\sum_{i=j}^{n-1}\log\frac{A_i}{1-A_i}
\ =\ c_2-\sum_{i=j}^{n-1}\xi_i,
\end{eqnarray*}
we get
\begin{eqnarray*}
\prob(Z_n>m\mid C_B(k,n),\ \mathcal A) &\ge&
1-\sum_{j=k+1}^{n-1} e^{-m\delta e^{\varepsilon(n-1-j)}},
\end{eqnarray*}
for any sequence of $\xi$'s such that
\begin{eqnarray*}
c_2-\sum_{i=j}^{n-1}\xi_i &\ge& 0\quad\mbox{for all }j\in[k,n -1],
\end{eqnarray*}
which is the case on $C_S(c_2,\varepsilon,k,n-1)$ and hence on $C_A(k,n)$,
as follows from the first inequality in \eqref{lln}
for all $\varepsilon\in(0,-\exptn\xi)$.
So, we have shown \eqref{l.b.ind.conv}, and the proof of the first lower bound
in Theorem \ref{thm:tail:generation:size.lower} is complete.

The lower limit for the stationary distribution follows similar arguments
if we start with an analogue of \eqref{Zn.A.C},
\begin{eqnarray}\label{Zn.A.C.infty}
\prob(Z>m) &=& \lim_{n\to\infty}\prob(Z_n>m)\nonumber\\
&\ge& \lim_{n\to\infty}\exptn[\prob(Z_n>m\mid \mathcal A);\ C_A(n)].
\end{eqnarray}
Then, similar to \eqref{lower.for.union.final},
we may use the fact that $F_I$ is long-tailed to conclude that
\begin{eqnarray}\label{lower.for.union.final.infty}
\lim_{n\to\infty}\prob(C_A(n))
&\ge& \frac{g(c_2)+o(1)}{2\varepsilon+a} \overline F_I(\log m)
\quad\mbox{as }m\to\infty,
\end{eqnarray}
which together with \eqref{Zn.A.C.1} justifies the lower bound
for the stationary tail distribution.

\section{Proof of the upper bound, Theorem~\ref{thm:tail:generation:size.upper}}
\label{sec:upper}

Let $W_n$ be a branching process without immigration, that is, $W_0=1$ and
\begin{eqnarray*}
W_{n+1} &=& \sum_{i=0}^{W_n}B_{n+1,i}\quad\mbox{for }n\ge 0.
\end{eqnarray*}

Let $W_n^{(1)}$ be the number of particles in $Z_n$ generated by the
immigrant arriving at time $1$, $W_n^{(2)}$ be the number of particles in $Z_n$
generated by the immigrant arriving at time $2$ and so on.
All these processes extinct in a finite time and are independent being
conditioned on the environment $\mathcal A$. In addition, $W_n^{(k)}$ has 
the same distribution with $W_{n-k}$ given the same success probabilities.
By the definition of $Z_n$,
\begin{eqnarray*}
Z_n &=& W_n^{(1)}+W_n^{(2)}+\ldots+W_n^{(n)},
\end{eqnarray*}
and hence, for any fixed $\varepsilon>0$,
\begin{eqnarray*}
\prob(Z_n>m) &\le&
\prob\bigl(W_n^{(k)}>me^{-\varepsilon(n-k)}(1-e^{-\varepsilon})
\mbox{ for some }k\le n\bigr)\\
&=& \exptn\bigl[\prob\bigl(W_n^{(k)}>me^{-\varepsilon(n-k)}(1-e^{-\varepsilon})
\mbox{ for some }k\le n\mid\mathcal A\bigr)\bigr].
\end{eqnarray*}
Splitting the area of integration into two parts, we get the following upper bound
\begin{eqnarray}\label{upperboundW}
\prob(Z_n>m)
&\le& \prob(S_{k,n-1}>\log m-\sqrt{\log m}-2\varepsilon(n-k) 
\mbox{ for some }k\in[0,n-1])\nonumber\\
&& + \exptn\bigl[\prob\bigl(W_n^{(k)}>me^{-\varepsilon(n-k)}(1-e^{-\varepsilon})
\mbox{ for some }k\le n\mid\mathcal A\bigr);\nonumber\\
&&\hspace{20mm}S_{k,n-1}\le\log m-\sqrt{\log m}-2\varepsilon(n-k) 
\mbox{ for all }k\in[0,n-1]\bigr].
\end{eqnarray}

Using \eqref{stab.cond} and strong subexponentiality of  $F$
we conclude that
\begin{eqnarray}\label{upper.1}
\lefteqn{\prob(S_{k,n-1}+2\varepsilon(n-k)>\log m-\sqrt{\log m}
\mbox{ for some }k\in[0,n-1])}\nonumber\\
&&\hspace{40mm}\sim\ \frac{1}{a-2\varepsilon}
\int_{\log m-\sqrt{\log m}}^{\log m-\sqrt{\log m}+n(a-2\varepsilon)}
\overline F(x)dx
\end{eqnarray}
as $m\to\infty$ uniformly for all $n$, see \cite{Dima2002} and also
\cite{FKZ}, Theorem 5.3.

Further, by the Markov inequality,
\begin{eqnarray*}
\prob\bigl(W_n^{(k)}>me^{-\varepsilon(n-k)}(1-e^{-\varepsilon})\mid\mathcal A\bigr) &\le&
\frac{\exptn(W_n^{(k)}\mid\mathcal A)}{me^{-\varepsilon(n-k)}(1-e^{-\varepsilon})}\\
&=& \frac{e^{S_{k,n-1}}}{me^{-\varepsilon(n-k)}(1-e^{-\varepsilon})}.
\end{eqnarray*}
Hence, on the event $\{S_{k,n-1}\le\log m-\sqrt{\log m}-2\varepsilon(n-k) 
\mbox{ for all }k\in[0,n-1]\}$
we have
\begin{eqnarray*}
\prob\bigl(W_n^{(k)}>me^{-\varepsilon(n-k)}(1-e^{-\varepsilon})\mid\mathcal A\bigr)
&\le& \frac{e^{-\varepsilon(n-k)}}{e^{\sqrt{\log m}}(1-e^{-\varepsilon})},
\end{eqnarray*}
which implies that
\begin{eqnarray}\label{upper.2}
\lefteqn{\exptn\bigl[\prob\bigl(W_n^{(k)}>m(1-\varepsilon)^{n-k}\varepsilon
\mbox{ for some }k\le n\mid\mathcal A\bigr);}\nonumber\\
&&\hspace{40mm}S_{k,n-1}\le\log m-\sqrt{\log m}-2\varepsilon(n-k) 
\mbox{ for all }k\in[0,n-1]\bigr]
\nonumber\\
&&\hspace{20mm}\le\ \frac{1}{e^{\sqrt{\log m}}(1-e^{-\varepsilon})}
\sum_{k=0}^\infty e^{-\varepsilon(n-k)}\nonumber\\
&&\hspace{20mm}=\ \frac{1}{e^{\sqrt{\log m}}(1-e^{-\varepsilon})^2}.
\end{eqnarray}
Substituting \eqref{upper.1} and \eqref{upper.2} into \eqref{upperboundW}, 
we deduce that, uniformly for all $n\ge 1$,
\begin{eqnarray*}
\prob(Z_n>m)
&\le& \frac{1+o(1)}{a-2\varepsilon}
\int_{\log m-\sqrt{\log m}}^{\log m-\sqrt{\log m}+na} \overline F(x)dx
 + \frac{1}{e^{\sqrt{\log m}}(1-e^{-\varepsilon})^2}.
\end{eqnarray*}
By the condition \eqref{cond.sqrt}, $\overline F(\log m-\sqrt{\log m})\sim\overline F(\log m)$
and $\overline F(\log m)e^{\sqrt{\log m}}\to\infty$ as $m\to\infty$, hence
\begin{eqnarray*}
\prob(Z_n>m)
&\le& \frac{1+o(1)}{a-2\varepsilon}
\int_{\log m}^{\log m+na} \overline F(x)dx,
\end{eqnarray*}
uniformly for all $n\ge 1$. Due to the arbitrary choice of $\varepsilon>0$, 
the proof of the upper bound \eqref{eq:stat1.eq} is complete.

The above arguments can be streamlined if we made use of the link
\eqref{perp} to stochastic difference equations.
Indeed, conditioning on the environment leads to 
\begin{eqnarray*}
\prob(Z_n>m) &=& \exptn\bigl[\prob(Z_n>m\mid\mathcal A)\bigr]\\
&\le& \prob\bigl[\exptn(Z_n\mid\mathcal A)>me^{-\sqrt{\log m}}\bigr]\\
&& + \exptn\bigl[\prob\bigl(Z_n>m\mid\mathcal A\bigr);
\ \exptn(Z_n\mid\mathcal A)\le me^{-\sqrt{\log m}}\bigr].
\end{eqnarray*}
For the first term on the right hand side we apply the asymptotics \eqref{perp1}.
To estimate of the second term, we can apply the Markov inequality to get
\begin{eqnarray*}
\prob\bigl(Z_n>m\mid\mathcal A\bigr) &\le&
\frac{\exptn(Z_n\mid\mathcal A)}{m}\\
&\le& \frac{me^{-\sqrt{\log m}}}{m}\ =\ e^{-\sqrt{\log m}}
\end{eqnarray*}
on the event $\exptn(Z_n\mid\mathcal A)\le me^{-\sqrt{\log m}}$
which completes the proof.

The proof of the stationary upper bound \eqref{eq:stat2.eq}
follows similar arguments with initial upper bound
\begin{eqnarray*}
\prob(Z>m) &=& \lim_{n\to\infty}\prob(Z_n>m)\\
&\le& \lim_{n\to\infty}\prob\bigl[\exptn(Z_n\mid\mathcal A)>me^{-\sqrt{\log m}}\bigr]\\
&& + \lim_{n\to\infty}\exptn\bigl[\prob\bigl(Z_n>m\mid\mathcal A\bigr);
\ \exptn(Z_n\mid\mathcal A)\le me^{-\sqrt{\log m}}\bigr].
\end{eqnarray*}
and further use of the asymptotics \eqref{perp2} instead of \eqref{perp1}
which is valid due to subexponentiality of the integrated tail distribution $F_I$.
The proof of Theorem \ref{thm:tail:generation:size.upper} is complete.

%

\section{Proof of the principle of a single atypical environment,
Theorem~\ref{thm:PSLE}}
\label{sec:PSLE}

As follows from the arguments presented in Section \ref{sec:lower},
for any fixed $c$ and $\varepsilon>0$,
\begin{eqnarray*}
\prob\biggl(\bigcup_{k=0}^{n-1} E_n^{(k)}(m)\biggr) &\sim&
\frac{1}{a+\varepsilon}
\int_{\log m}^{\log m+(a+\varepsilon)n} \overline F(x)dx\\
&\ge& \frac{1}{a+\varepsilon} \int_{\log m}^{\log m+an} \overline F(x)dx
\end{eqnarray*}
and the event presented on the left hand side implies $Z_n>m$ with high probability, that is,
\begin{eqnarray*}
\prob\biggl(Z_n>m\ \Big|\ \bigcup_{k=0}^{n-1} E_n^{(k)}(m)\biggr) &\to& 1
\quad\mbox{as }m\to\infty\mbox{ uniformly for all }n.
\end{eqnarray*}
Then the equality
\begin{eqnarray*}
\prob\biggl(\bigcup_{k=0}^{n-1} E_n^{(k)}(m)\ \Big|\ Z_n>m\biggr) &=&
\prob\biggl(Z_n>m\ \Big|\ \bigcup_{k=0}^{n-1} E_n^{(k)}(m)\biggr)
\frac{\prob\Bigl(\bigcup_{k=0}^{n-1} E_n^{(k)}(m)\Bigr)}{\prob(Z_n>m)}
\end{eqnarray*}
and Theorem \ref{thm:tail:generation:size.upper} imply that
\begin{eqnarray*}
\lim_{m\to\infty}\inf_{n}
\prob\biggl(\bigcup_{k=0}^{n-1} E_n^{(k)}(m)\ \Big|\ Z_n>m\biggr) &\ge&
\frac{a}{a+\varepsilon}.
\end{eqnarray*}
Letting $\varepsilon\downarrow 0$ concludes the proof.

\section{Related models}
\label{sec:extensions}

The techniques developed in this paper may be applied
to analysing a variety of similar models. We mention here a few of them.

{\bf Random-size immigration.} One may replace size-1 immigration by a
{\it random-size-im\-migra\-tion} where random sizes are i.i.d. and
independent of everything else, with a common light-tailed distribution 
(or, more generally, the sizes may be stochastically bounded by a random variable
with a light-tailed distribution).

A branching process $\{\widehat{Z}_n, n\ge 0\}$ with {\it state-dependent size-1 immigration} is a particular case here: an immigrant arrives only when the previous generation produces no offspring:
\begin{align*}
{\widehat Z}_{n+1}= \sum_{i=1}^{\max (1,{\widehat Z}_n)} B_{n+1,i}, \quad n\ge 0.
\end{align*}
Clearly, $\widehat{Z}_n \le Z_n$ a.s., for any $n$. Moreover, one can show that, for each $n$, the low bounds for $\prob(Z_n>m)$
and $\prob(\widehat{Z}_n>m)$ are asymptotically equivalent.
Then, in particular, the statement of Theorem \ref{thm:tail:fixed:generation:size} stays valid with $\widehat{Z}_n$ in place of $Z_n$.

{\bf Continuous-space analogue.}
Instead of the recursion \eqref{Zn},
one may consider a ``continuous-space'' recursion of the form
\begin{align*}
Z_{n+1} = Y_{n+1} + \int_0^{Z_n} dB_{n+1}(t)
\end{align*}
where $B_n$ are  subordinators with a light-tailed distribution of the
Levy measure (that depends on random parameters) and $\{Y_n\}$ are i.i.d. 
``innovations'' with a light-tailed distribution. 
A similar problem for a branching process  with immigration, 
but without random environment has been studied in a recent paper 
by \cite{FM}.


\end{document}